\documentclass{amsart}
\usepackage{amssymb}
\usepackage{amsmath}
\usepackage{amsfonts}

\newtheorem{theorem}{Theorem}
\theoremstyle{plain}

\newtheorem{corollary}{Corollary}

\newtheorem{definition}{Definition}

\newtheorem{lemma}{Lemma}

\newtheorem{remark}{Remark}

\numberwithin{equation}{section}
\input{tcilatex}

\begin{document}
\title[]{S\lowercase {ome} I\lowercase {ntegral} I\lowercase {nequalities} F%
\lowercase {or} F\lowercase {unctions} W\lowercase {hoose} S\lowercase
{econd} D\lowercase {erivatives} A\lowercase {re} $\varphi -$C\lowercase
{onvex} B\lowercase {y} U\lowercase {sing} F\lowercase {ractional} I%
\lowercase {ntegrals}}
\author{M. E\lowercase {sra} YILDIRIM}
\address{[Department of Mathematics, Faculty of Science, University of
Cumhuriyet, 58140, Sivas, Turkey}
\email{mesra@cumhuriyet.edu.tr}
\author{A\lowercase {bdullah} AKKURT}
\address{[Department of Mathematics, Faculty of Science and Arts, University
of Kahramanmara\c{s} S\"{u}t\c{c}\"{u} \.{I}mam, 46100, Kahramanmara\c{s},
Turkey}
\email{abdullahmat@gmail.com}
\author{H\lowercase {üseyin} YILDIRIM}
\address{[Department of Mathematics, Faculty of Science and Arts, University
of Kahramanmara\c{s} S\"{u}t\c{c}\"{u} \.{I}mam, 46100, Kahramanmara\c{s},
Turkey}
\email{hyildir@ksu.edu.tr}
\thanks{M.E. Yildirim was partially supported by the Scientific and
Technological Research Council of Turkey (TUBITAK Programme 2228-B)}

\begin{abstract}
In this paper, we obtain new estimates on generalization of
Hermite-Hadamard, Simpson and Ostrowski type inequalities for functions
whose second derivatives is $\varphi -$convex via fractional integrals.
\end{abstract}

\maketitle

\section{Introduction}

The following inequality is called Hermite-Hadamard Inequality;%
\begin{equation}
\begin{array}{c}
f\left( \dfrac{a+b}{2}\right) \leq \dfrac{1}{b-a}\int\limits_{a}^{b}f(x)dx%
\leq \dfrac{f(a)+f(b)}{2},%
\end{array}
\label{1.2}
\end{equation}%
where $f:I\subseteq 
%TCIMACRO{\U{211d} }%
%BeginExpansion
\mathbb{R}
%EndExpansion
\rightarrow 
%TCIMACRO{\U{211d} }%
%BeginExpansion
\mathbb{R}
%EndExpansion
$ is a convex function and $a,b\in I\ $with $a<b$. If $f$ is concave, then
both inequalities hold in the reversed direction .

The inequality \ref{1.2} inequality was first discovered by Hermite in 1881
in the Journal Mathesis. This inequality was known as Hermite-Hadamard
Inequlity, because this inequality was found by Mitrinovic Hermite and
Hadamard' note in Mathesis in 1974.

The inequality \ref{1.2} is studied by many authors, see (\cite{Be48}-\cite%
{Is13-2}, \cite{KAO11}-\cite{OAK-archiv}, \cite{Pa14}, \cite{SO12}-\cite%
{TY-archiv}) where further references are listed.

Firstly, we need to recall some concepts of convexity concerning our work.

\begin{definition}
\cite{Is13-1} A function $f:I\subset 
%TCIMACRO{\U{211d} }%
%BeginExpansion
\mathbb{R}
%EndExpansion
\rightarrow 
%TCIMACRO{\U{211d} }%
%BeginExpansion
\mathbb{R}
%EndExpansion
$ is said to be convex on $I\ $if inequality 
\begin{equation}
\begin{array}{c}
f(ta+(1-t)b)\leq tf(a)+(1-t)f(b),%
\end{array}
\label{1.1}
\end{equation}%
holds for all $a,b\in I\ $and $t\in \lbrack 0,1].$
\end{definition}

\begin{definition}
\cite{Is-15} Let $s\in (0,1].\ $A function $f:I\subseteq 
%TCIMACRO{\U{211d} }%
%BeginExpansion
\mathbb{R}
%EndExpansion
_{0}=[0,\infty )\rightarrow 
%TCIMACRO{\U{211d} }%
%BeginExpansion
\mathbb{R}
%EndExpansion
$ is said to be $s-$convex in the second sense if%
\begin{equation}
\begin{array}{c}
f(ta+(1-t)b)\leq t^{s}f(a)+(1-t)^{s}f(b),%
\end{array}
\label{1.3}
\end{equation}%
holds for all $a,b\in I\ $and $t\in \lbrack 0,1].$
\end{definition}

Tun\c{c} and Yildirim in \cite{TY-archiv}\ introduced the following
defination as follows:

\begin{definition}
\textit{A function} $f$:$I\subseteq \mathbb{R}\rightarrow \mathbb{R}$ 
\textit{is said to belong to the class of} $MT\left( I\right) $ \textit{if
it is nonnegative and for all} $x,y\in I$ \textit{and} $t\in \left(
0,1\right) $ \textit{satisfies the inequality};%
\begin{equation*}
f\left( tx+\left( 1-t\right) y\right) \leq \frac{\sqrt{t}}{2\sqrt{1-t}}%
f\left( x\right) +\frac{\sqrt{1-t}}{2\sqrt{t}}f\left( y\right) .
\end{equation*}
\end{definition}

Dragomir in \cite{Dr15} introduced the following defination as follows:

\begin{definition}
\cite{Dr15} Let\textit{\ }$\varphi :\left( 0,1\right) \rightarrow \left(
0,\infty \right) $ be a measurable function. We say that the function $%
f:I\rightarrow \left[ 0,\infty \right) $ \textit{is a }$\varphi -$convex
function on the interval $I$ if $x,y\in I$ we have%
\begin{equation*}
f\left( tx+\left( 1-t\right) y\right) \leq t\varphi \left( t\right) f\left(
x\right) +\left( 1-t\right) \varphi \left( 1-t\right) f\left( y\right) .
\end{equation*}
\end{definition}

\begin{remark}
According to defination $4$ for the special choose of $\varphi $ we can
obtain following

If we take $\varphi (t)\equiv 1$, we obtain classical convex.

If we take $\varphi (t)=t^{s-1}$, we obtain $s-$convex.

If we take $\varphi (t)=\dfrac{1}{2\sqrt{t}\sqrt{1-t}}$, we obtain $MT-$%
convex.
\end{remark}

Now, we will\ give some definitions and notations of fractional calculus
theory which are used later in this paper. Samko et al. in \cite{Sam} used
following definations as follows:

\begin{definition}
\label{d1}\cite{Sam} The Riemann-Liouville fractional integrals $%
J_{a^{+}}^{\alpha }f$ and $J_{b^{-}}^{\alpha }f$ of order $\alpha >0$ with $%
a\geq 0$ are defined by%
\begin{equation}
J_{a^{+}}^{\alpha }f(x)=\frac{1}{\Gamma \left( \alpha \right) }%
\dint\limits_{a}^{x}\left( x-t\right) ^{\alpha -1}f(t)dt,\ x>a  \label{1.4}
\end{equation}%
and%
\begin{equation}
J_{b^{-}}^{\alpha }f(x)=\frac{1}{\Gamma \left( \alpha \right) }%
\dint\limits_{x}^{b}\left( t-x\right) ^{\alpha -1}f(t)dt,\ x<b  \label{1.5}
\end{equation}%
where $f\in L_{1}\left[ a,b\right] $, respectively. Note that, $\Gamma
\left( \alpha \right) \ $is the Gamma function and $%
J_{a^{+}}^{0}f(x)=J_{b^{-}}^{\alpha }f(x)=f(x)$.
\end{definition}

\begin{definition}
\cite{Sam} The Euler Beta function is defined as follows:%
\begin{equation*}
\beta \left( x,y\right) =\dint\limits_{0}^{1}t^{x-1}\left( 1-t\right)
^{y-1}dt,\text{ }x,y>0.
\end{equation*}%
The incomplate beta function is defined as follows:%
\begin{equation*}
\beta \left( a,x,y\right) =\dint\limits_{0}^{a}t^{x-1}\left( 1-t\right)
^{y-1}dt,\text{ }x,y>0,\text{ }0<\alpha <1.
\end{equation*}
\end{definition}

\section{Main results}

Througout this paper, we use $S_{f}$ as follows;%
\begin{equation*}
\begin{array}{ll}
S_{f}\left( x,\lambda ,\alpha ;a,b\right) & \equiv \left( 1-\lambda \right)
\left \{ \frac{\left( b-x\right) ^{\alpha +1}-\left( x-a\right) ^{\alpha +1}%
}{b-a}\right \} f^{\prime }\left( x\right) \\ 
&  \\ 
& +\left( 1+\alpha -\lambda \right) \left \{ \frac{\left( x-a\right)
^{\alpha }+\left( b-x\right) ^{\alpha }}{b-a}\right \} f\left( x\right) \\ 
&  \\ 
& +\lambda \left \{ \frac{\left( x-a\right) ^{\alpha }(f\left( a\right)
+\left( b-x\right) ^{\alpha }f\left( b\right) }{b-a}\right \} \\ 
&  \\ 
& -\frac{\Gamma \left( \alpha +2\right) }{b-a}\left \{ J_{x^{-}}^{\alpha
}f\left( a\right) +J_{x^{+}}^{\alpha }f\left( b\right) \right \} ,%
\end{array}%
\end{equation*}%
for any $x\in \left[ a,b\right] ,$ $\lambda \in \left[ 0,1\right] $ and $%
\alpha >0.$

In (\cite{Pa15}), Jaekeun Park established the following lemma which is
necessary to prove our main results:

\begin{lemma}
\textit{Let} $f$: $I\subseteq 
%TCIMACRO{\U{211d} }%
%BeginExpansion
\mathbb{R}
%EndExpansion
\rightarrow 
%TCIMACRO{\U{211d} }%
%BeginExpansion
\mathbb{R}
%EndExpansion
$ \textit{be a twice differentiable function on the interior} $I^{0}$ 
\textit{of an interval }$I$\textit{\ such that} $f^{\prime \prime }\in L_{1}%
\left[ a,b\right] $, \textit{where} $a,b\in I$ \textit{with} $a<b$. \textit{%
Then, for any} $x\in \left[ a,b\right] $, $\lambda \in \left[ 0,1\right] $ 
\textit{and} $\alpha >0$ \textit{we have}%
\begin{equation*}
\begin{array}{cc}
S_{f}\left( x,\lambda ,\alpha ;a,b\right) & =\frac{\left( x-a\right)
^{\alpha +2}}{b-a}\int_{0}^{1}t\left( \lambda -t^{\alpha }\right) f^{\prime
\prime }\left( tx+\left( 1-t\right) a\right) dt \\ 
&  \\ 
& +\frac{\left( b-x\right) ^{\alpha +2}}{b-a}\int_{0}^{1}t\left( \lambda
-t^{\alpha }\right) f^{\prime \prime }\left( tx+\left( 1-t\right) b\right)
dt.%
\end{array}%
\end{equation*}
\end{lemma}

\begin{theorem}
Let\textit{\ }$\varphi :\left( 0,1\right) \rightarrow \left( 0,\infty
\right) $ be a measurable function. Assume also that $f:I\subset \lbrack
0,\infty )\rightarrow 
%TCIMACRO{\U{211d} }%
%BeginExpansion
\mathbb{R}
%EndExpansion
$ \textit{be a twice differentiable function on the interior} $I^{0}$ 
\textit{of an interval }$I$\textit{\ such that} $f^{\prime \prime }\in L_{1}%
\left[ a,b\right] $, \textit{where} $a,b\in I^{0}$ \textit{with} $a<b$. 
\textit{If} $\left\vert f^{\prime \prime }\right\vert ^{q}$ \textit{is }$%
\varphi -$\textit{convex on} $\left[ a,b\right] $ \textit{for some fixed} $%
q\geq 1$, \textit{then for any} $x=ta+\left( 1-t\right) b,t\in \left[ 0,1%
\right] ,\lambda \in \left[ 0,1\right] $, \textit{and} $\alpha >0$:%
\begin{equation}
\begin{array}{l}
\left\vert S_{f}\left( x,\lambda ,\alpha ,t,\varphi ;a,b\right) \right\vert
\leq A_{1}^{1-\frac{1}{q}}\left( \alpha ,\lambda \right) \left[ \frac{\left(
x-a\right) ^{\alpha +2}}{b-a}\left\{ A_{2}(\alpha ,\lambda ,t,\varphi
)\left\vert f^{\prime \prime }\left( x\right) \right\vert ^{q}\right. \right.
\\ 
\\ 
\left. +A_{3}\left( \alpha ,\lambda ,t,\varphi \right) \left\vert f^{\prime
\prime }\left( a\right) \right\vert ^{q}\right\} ^{\frac{1}{q}} \\ 
\\ 
\left. +\frac{\left( b-x\right) ^{\alpha +2}}{b-a}\left\{ A_{2}\left( \alpha
,\lambda ,t,\varphi \right) \left\vert f^{\prime \prime }\left( x\right)
\right\vert ^{q}+A_{3}\left( \alpha ,\lambda ,t,\varphi \right) \left\vert
f^{\prime \prime }\left( b\right) \right\vert ^{q}\right\} ^{\frac{1}{q}}%
\right] ,%
\end{array}
\label{1.6}
\end{equation}%
\textit{the above inequality for fractional integrals holds}, \textit{where}
\end{theorem}

\begin{equation*}
\begin{array}{ll}
A_{1}\left( \alpha ,\lambda \right) & =\frac{\alpha \lambda ^{1+\frac{2}{%
\alpha }}+1}{\alpha +2}-\frac{\lambda }{2}, \\ 
A_{2}\left( \alpha ,\lambda ,t,\varphi \right) & =\int_{0}^{1}\left \vert
t\left( \lambda -t^{\alpha }\right) \right \vert t\varphi \left( t\right) dt,
\\ 
A_{3}\left( \alpha ,\lambda ,t,\varphi \right) & =\int_{0}^{1}\left \vert
t\left( \lambda -t^{\alpha }\right) \right \vert \left( 1-t\right) \varphi
\left( 1-t\right) dt.%
\end{array}%
\end{equation*}

\begin{proof}
By using\ Lemma 1, the power mean inequality, then we get%
\begin{equation}
\begin{array}{l}
\left\vert S_{f}\left( x,\lambda ,\alpha ,t,\varphi ;a,b\right) \right\vert
\\ 
\leq \frac{\left( x-a\right) ^{\alpha +2}}{b-a}\left( \int_{0}^{1}\left\vert
t\left( \lambda -t^{\alpha }\right) \right\vert dt\right) ^{1-\frac{1}{q}%
}\left( \int_{0}^{1}\left\vert t\left( \lambda -t^{\alpha }\right)
\right\vert \left\vert f^{\prime \prime }\left( tx+\left( 1-t\right)
a\right) \right\vert ^{q}dt\right) ^{\frac{1}{q}} \\ 
\\ 
+\frac{\left( b-x\right) ^{\alpha +2}}{b-a}\left( \int_{0}^{1}\left\vert
t\left( \lambda -t^{\alpha }\right) \right\vert dt\right) ^{1-\frac{1}{q}%
}\left( \int_{0}^{1}\left\vert t\left( \lambda -t^{\alpha }\right)
\right\vert \left\vert f^{\prime \prime }\left( tx+\left( 1-t\right)
b\right) \right\vert dt\right) ^{\frac{1}{q}} \\ 
\\ 
=A_{1}^{1-\frac{1}{q}}\left( \alpha ,\lambda \right) \left[ \frac{\left(
x-a\right) ^{\alpha +2}}{b-a}\left( \int_{0}^{1}\left\vert t(\lambda
-t^{\alpha })\right\vert \left\vert f^{\prime \prime }\left( tx+\left(
1-t\right) a\right) \right\vert ^{q}dt\right) ^{\frac{1}{q}}\right. \\ 
\\ 
\left. +\frac{\left( b-x\right) ^{\alpha +2}}{b-a}\left(
\int_{0}^{1}\left\vert t(\lambda -t^{\alpha })\right\vert \left\vert
f^{\prime \prime }\left( tx+\left( 1-t\right) b\right) \right\vert
^{q}dt\right) ^{\frac{1}{q}}\right] ,%
\end{array}
\label{1.7}
\end{equation}%
where%
\begin{equation*}
A_{1}\left( \alpha ,\lambda \right) =\int_{0}^{1}\left\vert t\left( \lambda
-t^{\alpha }\right) \right\vert dt=\left( \frac{\alpha \lambda ^{1+\frac{2}{%
\alpha }}+1}{\alpha +2}-\frac{\lambda }{2}\right) .
\end{equation*}

Since $\left \vert f^{\prime \prime }\right \vert ^{q}$ is $\varphi -$convex
on $\left[ a,b\right] $, we have%
\begin{equation}
\begin{array}{ll}
I_{1} & =\int_{0}^{1}\left \vert t\left( \lambda -t^{\alpha }\right) \right
\vert \left \vert f^{\prime \prime }\left( tx+\left( 1-t\right) a\right)
\right \vert ^{q}dt \\ 
&  \\ 
& \leq \int_{0}^{1}\left \vert t\left( \lambda -t^{\alpha }\right) \right
\vert \left \{ t\varphi \left( t\right) \left \vert f^{\prime \prime }\left(
x\right) \right \vert ^{q}+\left( 1-t\right) \varphi \left( 1-t\right) \left
\vert f^{\prime \prime }\left( a\right) \right \vert ^{q}\right \} dt \\ 
&  \\ 
& =A_{2}\left( \alpha ,\lambda ,t,\varphi \right) \left \vert f^{\prime
\prime }\left( x\right) \right \vert ^{q}+A_{3}\left( \alpha ,\lambda
,t,\varphi \right) \left \vert f^{\prime \prime }\left( a\right) \right
\vert ^{q},%
\end{array}
\label{1.8}
\end{equation}%
and similarly we can obtain%
\begin{equation}
\begin{array}{ll}
I_{2} & =\int_{0}^{1}\left \vert t\left( \lambda -t^{\alpha }\right) \right
\vert \left \vert f^{\prime \prime }\left( tx+\left( 1-t\right) b\right)
\right \vert ^{q}dt \\ 
&  \\ 
& \leq \int_{0}^{1}\left \vert t\left( \lambda -t^{\alpha }\right) \right
\vert \left \{ t\varphi \left( t\right) \left \vert f^{\prime \prime }\left(
x\right) \right \vert ^{q}+\left( 1-t\right) \varphi \left( 1-t\right) \left
\vert f^{\prime \prime }\left( b\right) \right \vert ^{q}\right \} dt \\ 
&  \\ 
& =A_{2}\left( \alpha ,\lambda ,t,\varphi \right) \left \vert f^{\prime
\prime }\left( x\right) \right \vert ^{q}+A_{3}\left( \alpha ,\lambda
,t,\varphi \right) \left \vert f^{\prime \prime }\left( b\right) \right
\vert ^{q}.%
\end{array}
\label{1.9}
\end{equation}%
By substituting (\ref{1.8}) and (\ref{1.9}) in (\ref{1.7}), we get%
\begin{equation*}
\begin{array}{l}
\left \vert S_{f}\left( x,\lambda ,\alpha ,t,\varphi ;a,b\right) \right \vert
\\ 
\\ 
\leq \left( \frac{\alpha \lambda ^{1+\frac{2}{\alpha }}+1}{\alpha +2}-\frac{%
\lambda }{2}\right) ^{1-\frac{1}{q}}\left[ \frac{\left( x-a\right) ^{\alpha
+2}}{b-a}\left \{ \left \vert f^{\prime \prime }\left( x\right) \right \vert
^{q}\int_{0}^{1}\left \vert t\left( \lambda -t^{\alpha }\right) \right \vert
t\varphi \left( t\right) dt\right. \right. \\ 
\\ 
\left. +\left \vert f^{\prime \prime }\left( a\right) \right \vert
^{q}\int_{0}^{1}\left \vert t\left( \lambda -t^{\alpha }\right) \right \vert
\left( 1-t\right) \varphi \left( 1-t\right) dt\right \} ^{^{\frac{1}{q}}} \\ 
\\ 
+\frac{\left( b-x\right) ^{\alpha +2}}{b-a}\left \{ \left \vert f^{\prime
\prime }\left( x\right) \right \vert ^{q}\int_{0}^{1}\left \vert t\left(
\lambda -t^{\alpha }\right) \right \vert t\varphi \left( t\right) dt\right.
\\ 
\\ 
\left. \left. +\left \vert f^{\prime \prime }\left( b\right) \right \vert
^{q}\int_{0}^{1}\left \vert t\left( \lambda -t^{\alpha }\right) \right \vert
\left( 1-t\right) \varphi \left( 1-t\right) dt\right \} ^{^{\frac{1}{q}}}%
\right] .%
\end{array}%
\end{equation*}%
Thus the proof is complated.
\end{proof}

\begin{corollary}
Let $\varphi \left( t\right) =1$ \textit{in Theorem 1, then we get the
following inequality}:%
\begin{equation*}
\begin{array}{l}
\left\vert S_{f}\left( x,\lambda ,\alpha ;a,b\right) \right\vert \\ 
\\ 
\leq \left( \frac{\alpha \lambda ^{1+\frac{2}{\alpha }}+1}{\alpha +2}-\frac{%
\lambda }{2}\right) ^{1-\frac{1}{q}}\left[ \frac{\left( x-a\right) ^{\alpha
+2}}{b-a}\left\{ A_{2}\left( \alpha ,\lambda \right) \left\vert f^{\prime
\prime }\left( x\right) \right\vert ^{q}+A_{3}\left( \alpha ,\lambda \right)
\left\vert f^{\prime \prime }\left( a\right) \right\vert ^{q}\right\} \right.
\\ 
\\ 
\left. +\frac{\left( b-x\right) ^{\alpha +2}}{b-a}\{A_{2}\left( \alpha
,\lambda \right) \left\vert f^{\prime \prime }\left( x\right) \right\vert
^{q}+A_{3}\left( \alpha ,\lambda \right) \left\vert f^{\prime \prime }\left(
b\right) \right\vert ^{q}\}\right] .%
\end{array}%
\end{equation*}%
Where%
\begin{equation*}
\begin{array}{ll}
A_{2}\left( \alpha ,\lambda \right) & =\int_{0}^{1}\left\vert t\left(
\lambda -t^{\alpha }\right) \right\vert tdt=\dfrac{3-\left( \alpha +3\right)
\lambda +2\alpha \lambda ^{1+\frac{3}{\alpha }}}{3\left( \alpha +3\right) }%
\end{array}%
\end{equation*}%
and%
\begin{equation*}
\begin{array}{ll}
A_{3}\left( \alpha ,\lambda \right) & =\int_{0}^{1}\left\vert t\left(
\lambda -t^{\alpha }\right) \right\vert \left( 1-t\right) dt \\ 
&  \\ 
& =\dfrac{\alpha \lambda ^{1+\frac{2}{\alpha }}}{\alpha +2}-\dfrac{2\lambda
^{1+\frac{3}{\alpha }}}{3\left( \alpha +3\right) }+\dfrac{\alpha \lambda }{6}%
-\dfrac{\alpha }{\left( \alpha +2\right) \left( \alpha +3\right) }.%
\end{array}%
\end{equation*}
\end{corollary}

\begin{corollary}
\textit{If we choose} $\varphi \left( t\right) =1$ and $x=\frac{a+b}{2}$ in 
\textit{Theorem 1, we can obtain the corollary 2.2, 2.3, 2.4 in }(\cite{Pa15}%
), respectively for $\lambda =\frac{1}{3}$, $\lambda =0$, $\lambda =1$.
\end{corollary}

\begin{corollary}
Let $\varphi \left( t\right) =t^{s-1}$\ \textit{in Theorem 1,} \textit{then }%
we have%
\begin{equation*}
\begin{array}{l}
\left \vert S_{f}\left( x,\lambda ,\alpha ,t,\varphi ;a,b\right) \right \vert
\\ 
\\ 
\leq \left( \frac{\alpha \lambda ^{1+\frac{2}{\alpha }}+1}{\alpha +2}-\frac{%
\lambda }{2}\right) ^{1-\frac{1}{q}}\left[ \frac{\left( x-a\right) ^{\alpha
+2}}{b-a}\left \{ \left \vert f^{\prime \prime }\left( x\right) \right \vert
^{q}A_{4}\left( \alpha ,\lambda ,s\right) +\left \vert f^{\prime \prime
}\left( a\right) \right \vert ^{q}A_{5}\left( \alpha ,\lambda ,t,\varphi
\right) \right \} ^{\frac{1}{q}}\right. \\ 
\\ 
\left. +\frac{\left( b-x\right) ^{\alpha +2}}{b-a}\left \{ \left \vert
f^{\prime \prime }\left( x\right) \right \vert ^{q}A_{4}\left( \alpha
,\lambda ,s\right) +\left \vert f^{\prime \prime }\left( b\right) \right
\vert ^{q}A_{5}\left( \alpha ,\lambda ,t,\varphi \right) \right \} ^{\frac{1%
}{q}}\right] .%
\end{array}%
\end{equation*}%
Where%
\begin{equation*}
\begin{array}{ll}
A_{4}\left( \alpha ,\lambda ,s\right) & =2\dfrac{\lambda ^{\frac{s+2}{\alpha 
}+1}}{s+2}-2\dfrac{\lambda ^{\frac{s+2}{\alpha }+1}}{\alpha +s+2}+\dfrac{1}{%
\alpha +s+2} \\ 
&  \\ 
A_{5}\left( \alpha ,\lambda ,t,\varphi \right) & =\lambda \beta \left(
\lambda ^{\frac{1}{\alpha }},2,s+1\right) -\beta \left( \lambda ^{\frac{1}{%
\alpha }},\alpha +2,s+1\right) \\ 
&  \\ 
& +\beta \left( 1-\lambda ^{\frac{1}{\alpha }},\alpha +2,s+1\right) -\lambda
\beta \left( 1-\lambda ^{\frac{1}{\alpha }},2,s+1\right) .%
\end{array}%
\end{equation*}
\end{corollary}

\begin{theorem}
Let\textit{\ }$\varphi :\left( 0,1\right) \rightarrow \left( 0,\infty
\right) $ be a measurable function. For $f:I\subset \lbrack 0,\infty
)\rightarrow 
%TCIMACRO{\U{211d} }%
%BeginExpansion
\mathbb{R}
%EndExpansion
$ \textit{be a twice differentiable function on the interior} $I^{0}$ assume
also that $f^{\prime \prime }\in L_{1}\left[ a,b\right] $, \textit{where} $%
a,b\in I^{0}$ \textit{with} $a<b$. \textit{If} $\left\vert f^{\prime \prime
}\right\vert ^{q}$ \textit{is }$\varphi -$\textit{convex on} $\left[ a,b%
\right] $ \textit{for some fixed} $q>1$ \textit{with} $\frac{1}{p}+\frac{1}{q%
}=1$, \textit{then for any} $x\in \left[ a,b\right] ,\lambda \in \left[ 0,1%
\right] $ \textit{and} $\alpha >0$ \textit{the following inequality holds}%
\begin{equation}
\begin{array}{l}
\left\vert S_{f}\left( x,\lambda ,\alpha ,t,\varphi ;a,b\right) \right\vert
\\ 
\\ 
\leq B^{\frac{1}{p}}\left( \alpha ,\lambda ,p\right) \left[ \frac{\left(
x-a\right) ^{\alpha +2}}{b-a}\left\{ \left( \left\vert f^{\prime \prime
}\left( x\right) \right\vert ^{q}+\left\vert f^{\prime \prime }\left(
a\right) \right\vert ^{q}\right) \int_{0}^{1}t\varphi \left( t\right)
dt\right\} ^{\frac{1}{q}}\right. \\ 
\\ 
\left. +\frac{\left( b-x\right) ^{\alpha +2}}{b-a}\left\{ \left( \left\vert
f^{\prime \prime }\left( x\right) \right\vert ^{q}+\left\vert f^{\prime
\prime }\left( b\right) \right\vert ^{q}\right) \int_{0}^{1}t\varphi \left(
t\right) dt\right\} ^{\frac{1}{q}}\right] \text{,}%
\end{array}
\label{1.10}
\end{equation}%
\textit{where}%
\begin{equation*}
\begin{array}{l}
B\left( \alpha ,\lambda ,p\right) =\frac{\lambda ^{\frac{1+p+\alpha p}{%
\alpha }}}{\alpha }\left\{ \Gamma \left( 1+p\right) \Gamma \left( \frac{%
1+p+\alpha }{\alpha }\right) \text{ }\left( _{2}F_{1}\left( 1,1+p,2+p+\frac{%
1+p}{\alpha },1\right) \right) \right. \\ 
\\ 
\ \ \ \ \ \ \ \ \ \ \ \ \ \ \ \ \ \ \ \ \ \ \ \ \ \ \ \ \ \ \left. +\beta
\left( 1+p,-\frac{1+p+\alpha p}{\alpha }\right) -\beta \left( \lambda ,1+p,-%
\frac{1+p+\alpha p}{\alpha }\right) \right\} \text{,}%
\end{array}%
\end{equation*}%
\textit{also, for} $0<b<c$ \textit{and} $\left\vert z\right\vert <1,\
{}_{2}F_{1}$ \textit{is hypergeometric function defined by} 
\begin{equation*}
_{2}F_{1}\left( a,b,c,z\right) =\frac{1}{\beta \left( b,c-b\right) }%
\int_{0}^{1}t^{b-1}\left( 1-t\right) ^{c-b-1}\left( 1-zt\right) ^{-a}dt.
\end{equation*}
\end{theorem}

\begin{proof}
By using Lemma 1 and the H\"{o}lder inequality, we have the below inequality%
\begin{equation}
\begin{array}{l}
\left\vert S_{f}\left( x,\lambda ,\alpha ,t,\varphi ;a,b\right) \right\vert
\\ 
\\ 
\leq \frac{\left( x-a\right) ^{\alpha +2}}{b-a}\left( \int_{0}^{1}\left\vert
t\left( \lambda -t^{\alpha }\right) \right\vert ^{p}dt\right) ^{\frac{1}{p}%
}\left( \int_{0}^{1}\left\vert f^{\prime \prime }\left( tx+\left( 1-t\right)
a\right) \right\vert ^{q}dt\right) ^{\frac{1}{q}} \\ 
\\ 
+\frac{\left( b-x\right) ^{\alpha +2}}{b-a}\left( \int_{0}^{1}\left\vert
t\left( \lambda -t^{\alpha }\right) \right\vert ^{p}dt\right) ^{\frac{1}{p}%
}\left( \int_{0}^{1}\left\vert f^{\prime \prime }\left( tx+\left( 1-t\right)
b\right) \right\vert ^{q}dt\right) ^{\frac{1}{q}} \\ 
\\ 
=\left( \int_{0}^{1}\left\vert t\left( \lambda -t^{\alpha }\right)
\right\vert ^{p}\right) ^{\frac{1}{p}}\left[ \frac{\left( x-a\right)
^{\alpha +2}}{b-a}\left( \int_{0}^{1}\left\vert f^{\prime \prime }\left(
tx+\left( 1-t\right) a\right) \right\vert ^{q}dt\right) ^{\frac{1}{q}}\right.
\\ 
\\ 
\left. +\frac{\left( b-x\right) ^{\alpha +2}}{b-a}\left(
\int_{0}^{1}\left\vert f^{\prime \prime }\left( tx+\left( 1-t\right)
b\right) \right\vert ^{q}dt\right) ^{\frac{1}{q}}\right] .%
\end{array}
\label{1.11}
\end{equation}%
Since $\left\vert f^{\prime \prime }\right\vert $ is $\varphi -$convex on $%
\left[ a,b\right] $, we have%
\begin{equation}
\begin{array}{ll}
\int_{0}^{1}\left\vert f^{\prime \prime }\left( tx+\left( 1-t\right)
a\right) \right\vert ^{q}dt & \leq \int_{0}^{1}t\varphi \left( t\right)
\left\vert f^{\prime \prime }\left( x\right) \right\vert ^{q}dt \\ 
&  \\ 
& +\int_{0}^{1}\left( 1-t\right) \varphi \left( 1-t\right) \left\vert
f^{\prime \prime }\left( a\right) \right\vert ^{q}dt \\ 
&  \\ 
& =\left( \left\vert f^{\prime \prime }\left( x\right) \right\vert
^{q}+\left\vert f^{\prime \prime }\left( a\right) \right\vert ^{q}\right)
\int_{0}^{1}t\varphi \left( t\right) dt,%
\end{array}
\label{1.12}
\end{equation}%
and using same tecnique, we get%
\begin{equation}
\begin{array}{ll}
\int_{0}^{1}\left\vert f^{\prime \prime }\left( tx+\left( 1-t\right)
b\right) \right\vert ^{q}dt & \leq \int_{0}^{1}t\varphi \left( t\right)
\left\vert f^{\prime \prime }\left( x\right) \right\vert ^{q}dt \\ 
&  \\ 
& +\int_{0}^{1}\left( 1-t\right) \varphi \left( 1-t\right) \left\vert
f^{\prime \prime }\left( b\right) \right\vert ^{q}dt \\ 
&  \\ 
& =\left( \left\vert f^{\prime \prime }\left( x\right) \right\vert
^{q}+\left\vert f^{\prime \prime }\left( b\right) \right\vert ^{q}\right)
\int_{0}^{1}t\varphi \left( t\right) dt.%
\end{array}
\label{1.13}
\end{equation}%
On the other hand we can obtain the following equality;%
\begin{equation}
\begin{array}{ll}
B\left( \alpha ,\lambda ,p\right) & =\int_{0}^{1}\left\vert t\left( \lambda
-t^{\alpha }\right) \right\vert ^{p}dt \\ 
&  \\ 
& =\int_{0}^{\lambda ^{\frac{1}{\alpha }}}\left\{ t(\lambda -t^{\alpha
})\right\} ^{p}dt+\int_{\lambda ^{\frac{1}{\alpha }}}^{1}\left\{ t\left(
t^{\alpha }-\lambda \right) \right\} ^{p}dt \\ 
&  \\ 
& =C_{1}\left( \alpha ,\lambda ,p\right) +C_{2}\left( \alpha ,\lambda
,p\right) .%
\end{array}
\label{1.14}
\end{equation}%
By letting $\lambda -t^{\alpha }=u$ and $t^{\alpha }=u$, respectively, we
have%
\begin{equation}
\begin{array}{ll}
C_{1}\left( \alpha ,\lambda ,p\right) & =\int_{0}^{\lambda ^{\frac{1}{\alpha 
}}}\left\{ t\left( \lambda -t^{\alpha }\right) \right\} ^{p}dt \\ 
&  \\ 
& =\frac{1}{\alpha }\int_{0}^{\lambda }u^{p}\left( \lambda -u\right) ^{\frac{%
1+p-\alpha }{\alpha }}du \\ 
&  \\ 
& =\frac{1}{\alpha }\int_{0}^{1}\lambda ^{p}y^{p}\lambda ^{\frac{1+p-\alpha 
}{\alpha }}\left( 1-y\right) ^{\frac{1-\alpha +p}{\alpha }}\lambda dy \\ 
&  \\ 
& =\dfrac{\lambda ^{\frac{p\alpha +1+p}{\alpha }}}{\alpha }%
\int_{0}^{1}y^{p}\left( 1-y\right) ^{\frac{1+p}{\alpha }}\left( 1-y\right)
^{-1}dy \\ 
&  \\ 
& =\frac{\lambda ^{\frac{1+p+\alpha p}{\alpha }}}{\alpha }\Gamma \left(
1+p\right) \Gamma \left( \frac{1+p+\alpha }{\alpha }\right) _{2}F_{1}\left(
1,1+p,2+p+\frac{1+p}{\alpha },1\right) ,%
\end{array}
\label{1.15}
\end{equation}%
and 
\begin{equation}
\begin{array}{ll}
C_{2}\left( \alpha ,\lambda ,p\right) & =\int_{\lambda ^{\frac{1}{\alpha }%
}}^{1}\left\{ t\left( t^{\alpha }-\lambda \right) \right\} ^{p}dt \\ 
&  \\ 
& =\frac{1}{\alpha }\int_{\lambda ^{u}}^{1}\frac{1+p-\alpha }{\alpha }\left(
u-\lambda \right) ^{p}du \\ 
&  \\ 
& =\frac{\lambda ^{\frac{1+p+\alpha p}{\alpha }}}{\alpha }\left\{ \beta
\left( 1+p,-\frac{1+p+\alpha p}{\alpha }\right) -\beta \left( \lambda ,1+p,-%
\frac{1+p+\alpha p}{\alpha }\right) \right\} .%
\end{array}
\label{1.16}
\end{equation}%
Thus, we get the desired result.
\end{proof}

\begin{corollary}
Let $\varphi \left( t\right) =1$\ \textit{in Theorem 2}, \textit{then we get
the following inequality} for \textit{any} $x\in \left[ a,b\right] ,\lambda
\in \left[ 0,1\right] $ \textit{and} $\alpha >0$\textit{;}%
\begin{equation*}
\begin{array}{l}
\left\vert S_{f}\left( x,\lambda ,\alpha ,t,\varphi ;a,b\right) \right\vert
\\ 
\leq \left( \int_{0}^{1}\left\vert t\left( \lambda -t^{\alpha }\right)
\right\vert ^{p}dt\right) ^{\frac{1}{p}}\left[ \frac{\left( x-a\right)
^{\alpha +2}}{b-a}\left\{ \frac{\left( \left\vert f^{\prime \prime }\left(
x\right) \right\vert ^{q}+\left\vert f^{\prime \prime }\left( a\right)
\right\vert ^{q}\right) }{2}\right\} ^{\frac{1}{q}}\right. \\ 
\left. +\frac{\left( b-x\right) ^{\alpha +2}}{b-a}\left\{ \frac{\left(
\left\vert f^{\prime \prime }\left( x\right) \right\vert ^{q}+\left\vert
f^{\prime \prime }\left( b\right) \right\vert ^{q}\right) }{2}\right\} ^{%
\frac{1}{q}}\right] .%
\end{array}%
\end{equation*}
\end{corollary}

\begin{corollary}
\textit{If we choose} $\varphi \left( t\right) =1$ and $x=\frac{a+b}{2}$ in 
\textit{Theorem 2, we can obtain the corollary 2.6, 2.7, 2.8 in }(\cite{Pa15}%
), respectively for $\lambda =\frac{1}{3}$, $\lambda =0$, $\lambda =1$.
\end{corollary}

\begin{corollary}
Let \textit{in }$\varphi \left( t\right) =t^{s-1}$\textit{\ Theorem} 2, 
\textit{then we obtain}%
\begin{equation*}
\begin{array}{l}
\left\vert S_{f}\left( x,\lambda ,\alpha ,t,\varphi ;a,b\right) \right\vert
\\ 
\leq \left( \int_{0}^{1}\left\vert t\left( \lambda -t^{\alpha }\right)
\right\vert ^{p}dt\right) ^{\frac{1}{p}}\left[ \frac{\left( x-a\right)
^{\alpha +2}}{b-a}\left\{ \frac{\left( \left\vert f^{\prime \prime }\left(
x\right) \right\vert ^{q}+\left\vert f^{\prime \prime }\left( a\right)
\right\vert ^{q}\right) }{s+1}\right\} ^{\frac{1}{q}}\right. \\ 
\left. +\frac{\left( b-x\right) ^{\alpha +2}}{b-a}\left\{ \frac{\left(
\left\vert f^{\prime \prime }\left( x\right) \right\vert ^{q}+\left\vert
f^{\prime \prime }\left( b\right) \right\vert ^{q}\right) }{s+1}\right\} ^{%
\frac{1}{q}}\right] \text{.}%
\end{array}%
\end{equation*}%
\newpage
\end{corollary}


\begin{thebibliography}{99}
\bibitem{Be48} E. F. Beckenbach, Convex functions, \textit{Bull. Amer. Math.
Soc.}, \textbf{54} (1948) 439-460.
http://dx.doi.org/10.1090/s0002-9904-1948-08994-7

\bibitem{Da10} Z. Dahmani, On Minkowski and Hermite-Hadamard integral
inequalities via fractional integration, \textit{Ann. Funct. Anal.}, 1(%
\textbf{1}) (2010) 51-58. http://dx.doi.org/10.15352/afa/1399900993

\bibitem{Dr15} S. S. Dragomir, Inequalities of Jensen type for $\varphi $%
-convex functions, \textit{Fasc. Math.} \textbf{55} (2015) 35-52

\bibitem{HuMa} H. Hudzik and L. Maligranda, Some remarks on s-convex
functions, Aequationes Math., 48 (1994), no. 1, 100-111.

\bibitem{IBN14} I. I\c{s}can, K. Bekar, S. Numan, Hermite-Hadamard an
Simpson type inequalities for differentiable quasi-geometrically convex
func- tions, \textit{Turkish }$J.$\textit{\ of Anal. and Number Theory}, 2(%
\textbf{2}) (2014) 42-46. http://dx.doi.org/10.12691/tjant-2-2-3

\bibitem{Is13-1} I. I\c{s}can, New estimates on generalization of some
integral inequalities for ds-convex functions and their applications, 
\textit{Int. J. Pure Appl. Math.}, 86(4) (2013) 727-746.
http://dx.doi.org/10.12732/ijpam.v86i4.11

\bibitem{Is13-2} I. I\c{s}can, Generalization of different type integral
inequalities via fractional integrals for functions whose second derivatives
absolute value are quasi-convex \textit{Konuralp Journal of Mathematics},
1(2) (2013) 67-79.

\bibitem{Is-15} I. I\c{s}can, On generalization of different type integral
inequalities for $s$-convex functions via fractional integrals presented

\bibitem{KAO11} H. Kavurmaci, M. Avci, M. E. \"{O}zdemir, New inequalities
of Hermite- Hadamard's type for convex functions with applications, \textit{%
Journ. of In}-\textit{equal. and Appl}., 2011:86 (2011).
http://dx.doi.org/10.1186/1029-242x-2011-86

\bibitem{Mi93} V. G. Mihesan, A generalization of the convexity, \textit{%
Seminar on Functional Equations, Approx. and Convex, Cluj-Napoca, Romania}
(1993).

\bibitem{OAK-archiv} M. E. \"{O}zdemir, M. Avic, H. Kavurmaci,
Hermite-Hadamard type inequalities for $s$-convex and $s$-concave functions
via fractional integrals, arXiv:1202.0380v1[math.CA].

\bibitem{Pa11} J. Park, Generalization of some Simpson-like type
inequalities via differentiable $s$-convex mappings in the second sense, 
\textit{In}- \textit{ter. J. of Math. and Math. Sci}., 2011 Art No: 493531,
13 pages. http://dx.doi.org/10.1155/2011/493531

\bibitem{Pa14} Jaekeun Park, Some new Hermite-Hadamard-like type
inequalities on ge- ometrically convex functions, \textit{Inter. J. of Math.
Anal}., 8(\textbf{16}) (2014),793-802.
http://dx.doi.org/10.12988/ijma.2014.4243

\bibitem{Pa15} Jaekeun Park, On Some Integral Inequalities for Twice
Differentiable Quasi-Convex and Convex Functions via Fractional Integrals, 
\textit{Applied Mathematical Sciences}, Vol. 9(62) (2015), 3057-3069 HIKARI
Ltd, www.m-hikari.com. http://dx.doi.org/10.12988/ams.2015.53248.

\bibitem{Sam} Samko, S.G., Kilbas A.A. and Marichev, O.I., Fractional
Integrals and Derivatives, Theory and Applications, \textit{Gordon and Breach%
}, 1993, ISBN 2881248640.

\bibitem{SO12} M. Z. Sarikaya, H. Ogunmez, On new inequalities via
Riemann-Liouville fractional integration, \textit{Abstract and applied
analysis}, 2012 (2012) 10 pages, Art ID:428983.
http://dx.doi.org/10.1155/2012/428983

\bibitem{SSYB11} M. Z. Sarikaya, E. Set, H. Yildiz, N. Basak, Hermite-
Hadamard's inequalities for fractional integrals and related frac- tional
inequalities, \textit{Math. and Comput. Model}., 2011 (2011).
http://dx.doi.org/10.1016/j.mcm.2011.12.048

\bibitem{SSO-archiv} E. Set, M. Z. Sarikaya, M. E. \"{O}zdemir, Some
Ostrowski's type Inequalities for functions whose second derivatives are $s$%
-convex in the second sense, arXiv:1\textit{006.24} 88v1 [\textit{math}. CA]
12 \textit{June 2010}.

\bibitem{SOSK} E. Set, E. Ozdemir, M. Z. Sarikaya, F. Karako,
Hermite-Hadamard type inequalities for mappings whose derivatives are $s$%
-convex in the second sense via fractional integrals, \textit{Khayyam J.
Math., }1(1) (2015) 62-70.

\bibitem{To88} Gh. Toader, On a generalization of the convexity, \textit{%
Mathematica}, 30(53) (1988), 83-87.

\bibitem{Tu11} M.Tunc, On some new inequalities for convex functions, 
\textit{Turk. J. Math}., \textbf{35} (2011), 1-7.

\bibitem{TY-archiv} M. Tunc, H. Yildirim, \textit{On MT-Convexity, arXiv:
1205.5453 [math. CA] 24 May 2012}
\end{thebibliography}
\end{document}